\documentclass[12pt,amsmath]{amsart}
\usepackage{graphicx} 
\usepackage{amsthm}

\usepackage{amssymb, enumitem,mathtools}
\usepackage{amscd}
\usepackage{amsmath}
\usepackage{ mathrsfs }
\usepackage{indentfirst}
\usepackage{amssymb}
\usepackage{enumitem}
\usepackage{color}
\usepackage{float}
\usepackage[breaklinks=true,colorlinks=true,linkcolor=blue,citecolor=red,urlcolor=green]{hyperref}

\newtheorem{theorem}{Theorem}[section]

\newtheorem{remark}[theorem]{Remark}

\newtheorem{proposition}[theorem]{Proposition}

\date{}

\begin{document}
\title[]{
Rigidity of compact four-dimensional weakly Einstein Ricci Solitons}

\author{JeongHyeong Park}
\address{Department of Mathematics, Sungkyunkwan University, Suwon 16419, Korea}
\email{parkj@skku.edu}
\author{Wooseok Shin}
\address{Department of Mathematics, Sungkyunkwan University, Suwon 16419, Korea}
\email{tlsdntjr@skku.edu}\
\maketitle
\begin{abstract} 
We prove that every compact four-dimensional weakly Einstein Ricci soliton is Einstein. The nontrivial compact case reduces to the
gradient shrinking setting, where a differential identity for weakly Einstein four-manifolds, together with the curvature identity for
gradient Ricci solitons, yields the pointwise relation $|R|^2\nabla f=0$ for the soliton potential $f$. Consequently, no compact proper 
weakly Einstein four-manifold admits a Ricci soliton structure. A noncompact homogeneous example shows that the compactness assumption is essential.
\end{abstract}
\vspace{0.4cm}

\section{Introduction}
Let $(M^n,g)$ be an $n$-dimensional Riemannian manifold, and let $R$, $\rho$, and $\tau$ denote its Riemann curvature tensor, Ricci tensor, and scalar curvature, respectively. Throughout the paper, repeated indices are summed over. We define the symmetric $(0,2)$-tensor $\check R$ by $$\check R_{ij}=R_{ikpq}R_j{}^{kpq}. $$ The manifold $(M^n,g)$ is called \emph{weakly Einstein} if $ \check R=\frac{|R|^2}{n}g. $

The weakly Einstein condition was introduced by Euh, Park, and
Sekigawa \cite{EPS2}, motivated by the four-dimensional curvature
identity established in \cite{EPS1}. They also characterized the
condition in terms of generalized Singer--Thorpe bases. In dimension four, the weakly Einstein condition has a useful equivalent tensorial formulation. Let $ e=\rho-\frac{\tau}{4}g $ be the trace-free Ricci tensor. Then $(M^4,g)$ is weakly Einstein if and only if $ 6W(e)=-\tau e, $ where $W$ denotes the Weyl tensor and $[W(e)]_{ij}=W_{iabj}e^{ab} $ with our curvature convention; see \cite{DEKP}. In particular, every Einstein four-manifold and every conformally flat scalar-flat four-manifold is weakly Einstein. We call a weakly Einstein four-manifold \emph{proper} if it belongs to neither of these two classes. More recently, Derdzinski, Park, and Shin \cite{DPS} obtained a complete classification of non-Einstein weakly Einstein algebraic curvature tensors on an oriented Euclidean four-space. The resulting algebraic-equivalence types form three disjoint five-dimensional families. The present paper concerns a different aspect of the weakly Einstein condition, namely its interaction with the Ricci soliton equation. A Ricci soliton is a Riemannian manifold $(M,g)$ satisfying $ \rho+\frac12\mathcal L_Xg=\lambda g $ for some vector field $X$ and some constant $\lambda$. It is called shrinking, steady, or expanding according as $\lambda>0$, $\lambda=0$, or $\lambda<0$, respectively. When $X=\nabla f$ for some smooth function $f$, the Ricci soliton equation becomes $ \rho+\operatorname{Hess}f=\lambda g, $ and $(M,g,f)$ is called a gradient Ricci soliton. 

The purpose of this paper is to investigate compact four-dimensional weakly Einstein manifolds admitting Ricci soliton structures. Our main result is the following rigidity theorem. 
\begin{theorem}\label{thm:main} Every compact four-dimensional weakly Einstein Ricci soliton is Einstein. Consequently, no compact four-dimensional proper weakly Einstein manifold admits a Ricci soliton structure. \end{theorem} 

The compactness assumption in Theorem~\ref{thm:main} is essential. In their classification of four-dimensional homogeneous weakly
Einstein manifolds, Arias-Marco and Kowalski \cite{AK} singled out a
nonlocally symmetric example, known as the \emph{EPS space}, which is
unique up to local homothety \cite{DPS2}.
It may be realized by a left-invariant metric
on a solvable Lie group of the form
$
\mathbb{R}\ltimes\mathbb{R}^{3}.
$
Garc\'ia-R\'io, Mari\~no-Villar, V\'azquez-Abal, and
V\'azquez-Lorenzo \cite[p.~5]{GR} observed that
$
Q+3\operatorname{Id}
$
is a derivation, where $Q$ denotes the Ricci operator. Thus, with the
normalization adopted there,
the EPS metric is an expanding
algebraic Ricci soliton with soliton constant $\lambda=-3$.
Although the numerical value of the soliton constant changes under
homothetic rescaling, the normalization-independent conclusion is
that the EPS space admits a non-Einstein expanding algebraic Ricci
soliton structure. Consequently, non-Einstein four-dimensional weakly
Einstein Ricci solitons do exist in the noncompact setting.

We are not aware of any previous rigidity result for compact four-dimensional weakly Einstein Ricci solitons. Every compact steady or expanding Ricci soliton is Einstein, while Perelman showed that every compact shrinking Ricci soliton metric admits a gradient shrinking Ricci soliton structure; see \cite[\S2.4 and Remark~3.2]{Perelman}. Consequently, the only nontrivial compact case reduces to the gradient shrinking setting. The key observation in the proof is the pointwise identity $ \check R(\nabla f,\cdot)=0, $ obtained by combining a differential consequence of the weakly Einstein condition with the curvature identity of a gradient Ricci soliton. Although Theorem~\ref{thm:main} concerns arbitrary compact Ricci solitons, the pointwise argument itself does not require compactness in the gradient case. In fact, it also shows that every connected four-dimensional weakly Einstein gradient Ricci soliton is Einstein, without any compactness or completeness assumption.

\section{Curvature conventions and basic identities}

Throughout the paper, all manifolds are assumed to be smooth,
connected, and without boundary. Our curvature tensor is defined by
$$
R(X,Y)Z
=
\nabla_X\nabla_YZ-\nabla_Y\nabla_XZ-\nabla_{[X,Y]}Z,
$$
and we write
$$
R(X,Y,Z,W)=g(R(X,Y)Z,W).
$$
Thus, with respect to a local frame,
$$
\rho_{ab}=g^{ij}R_{iabj},
\qquad
\tau=g^{ab}\rho_{ab}.
$$
We define
$$
\check R_{ij}=R_{ikpq}R_j{}^{kpq}.
$$

With our convention, the second Bianchi identity and its contracted
form are
$$
\nabla_iR_{jabc}
+
\nabla_jR_{aibc}
+
\nabla_aR_{ijbc}
=
0
$$
and
$$
\nabla^kR_{iabk}
=
\nabla_i\rho_{ab}-\nabla_a\rho_{ib},
$$
respectively.

If $(M,g,f)$ is a gradient Ricci soliton satisfying
$$
\rho+\operatorname{Hess}f=\lambda g,
$$
then, using the contracted second Bianchi identity above, differentiating the soliton equation and commuting covariant
derivatives give
\begin{equation}\label{eq:soliton-curvature}
\begin{aligned}
\nabla^kR_{iabk}
&=\nabla_i\rho_{ab}-\nabla_a\rho_{ib}\\
&=
-\nabla_i\nabla_a\nabla_bf
+\nabla_a\nabla_i\nabla_bf=
R_{iabk}\nabla^kf.
\end{aligned}
\end{equation}

\section{A differential identity and the proof of Theorem \ref{thm:main}}

{We first establish the following differential identity.
\begin{proposition}\label{31}

Let $(M^4,g)$ be a weakly Einstein manifold. Then
$$\
R_{iabj}\nabla^i\rho^{ab}=0,
\qquad
R_{iabj}\nabla^kR^{iab}{}_k=0.
$$\
\end{proposition}

\begin{proof}
Differentiating
$$
\check R_{ij}= R_{iabc} R_j{}^{abc}
$$
 and using the second Bianchi identity and the curvature symmetries, we obtain
\begin{align*}
\nabla^i\check R_{ij}
&=
(\nabla^iR_{iabc})R_j{}^{abc}
+
R_{iabc}\nabla^iR_j{}^{abc}\\
&=R_{jabc}\nabla^c\rho^{ab}
-
R_{jabc}\nabla^b\rho^{ac}
+
\frac14\nabla_j|R|^2
\\&
=\frac14\nabla_j|R|^2
+
2R_{iabj}\nabla^i\rho^{ab}.
\end{align*}\

Since $(M^4,g)$ is weakly Einstein,
$$\
\nabla^i\check R_{ij}
=
\frac14\nabla_j|R|^2.
$$\
Hence
$$\
R_{iabj}\nabla^i\rho^{ab}=0.
$$\
Moreover, the contracted second Bianchi identity gives
$$\
\nabla^kR_{iabk}
=
\nabla_i\rho_{ab}-\nabla_a\rho_{ib},
$$
and consequently
$$\
R_{iabj}\nabla^kR^{iab}{}_k
=
2R_{iabj}\nabla^i\rho^{ab}
=
0.
$$
\end{proof}
}

\begin{proof}[Proof of Theorem~\ref{thm:main}]
Every compact steady or expanding Ricci soliton is Einstein; see
\cite[\S2.4]{Perelman}. Hence it remains to consider the shrinking
case. By \cite[Remark~3.2]{Perelman}, a compact shrinking Ricci
soliton metric admits a gradient shrinking Ricci soliton structure.
Thus, there exist $f\in C^\infty(M)$ and a constant $\lambda>0$ such
that
$$
\rho+\operatorname{Hess}f=\lambda g.
$$
By \eqref{eq:soliton-curvature}, Proposition~\ref{31}, and the
curvature symmetries, we obtain

$$
\begin{aligned}
0
&=
R_{iabj}\nabla^kR^{iab}{}_k
=
R_{iabj}R^{iab}{}_{\ell}\nabla^\ell f =
\check R_{j\ell}\nabla^\ell f\\
&=
\frac14|R|^2\nabla_jf.
\end{aligned}
$$
Hence
$$
|R|^2\nabla f=0
$$
everywhere on $M$.
Set
$
U=\{x\in M:|R|(x)>0\}.
$
On $U$, we have $\nabla f=0$, and hence
$\operatorname{Hess}f=0$. The soliton equation therefore yields
$
\rho=\lambda g
$
on $U$. On $M\setminus U$, we have $R=0$, and consequently
$\rho=0$.

We claim that $U$ is closed. Let $x_\nu\in U$ and suppose that
$x_\nu\to x$. Since $\rho=\lambda g$ on $U$, continuity gives
$$
\rho_x=\lambda g_x.
$$
If $x\notin U$, then $R_x=0$, and hence $\rho_x=0$. Thus $\lambda g_x=0$, contradicting
$\lambda>0$. Thus $U$ is both open and closed. Since $M$ is
connected, either $U=M$ or $U=\varnothing$. In the first case
$\rho=\lambda g$, while in the second case $R=0$. In either case,
$g$ is Einstein.
\end{proof}
\begin{remark}
\emph{Compactness is used above only to reduce an arbitrary Ricci soliton
to the gradient shrinking case. For any four-dimensional weakly
Einstein gradient Ricci soliton, the same pointwise computation gives
$$
|R|^2\nabla f=0.
$$
If $\lambda=0$, then $\rho=0$ both on $\{|R|>0\}$ and on its
complement. If $\lambda\ne0$, the same open-and-closed argument
applies. Hence every connected four-dimensional weakly Einstein
gradient Ricci soliton is Einstein, without any compactness or
completeness assumption.}
\end{remark}
\vspace{0.4cm}

\section*{Acknowledgements}
This work was supported by the National Research Foundation of Korea (NRF)
grant funded by the Korea government (MSIT) (RS-2024-00334956).

\end{document}